\documentclass{elsarticle}

\usepackage[utf8]{inputenc}

\usepackage{graphicx}
\usepackage[utf8]{inputenc}
\usepackage{subfigure}
\usepackage{amsmath}
\usepackage{amsthm}
\usepackage{amssymb, mathtools, hyperref}
\usepackage{graphicx}
\usepackage{cleveref}
\usepackage{tikz}
\usetikzlibrary{patterns,shapes}
\usepackage{gastex}
\usepackage{todonotes}
\usepackage{cite}

\def\N{\mathbb N}
\def\A{\mathcal A}
\def\B{\mathcal B}

\def\LL{\mathcal L}

\def\uu{\mathbf{u}}
\def\aa{\mathbf{a}}
\def\bb{\mathbf{b}}
\def\vv{\mathbf{v}}

\def\colour{\mathrm{colour}}

\theoremstyle{definition}
\newtheorem{definition}{Definition}

\newtheorem{remark}[definition]{Remark}
\newtheorem{example}[definition]{Example}

\theoremstyle{plain}
\newtheorem{theorem}[definition]{Theorem}

\newtheorem{lemma}[definition]{Lemma}
\newtheorem{observation}[definition]{Observation}

\begin{document}

\title{Frequencies of letters in infinite $k$-balanced sequences}

\author[cvut]{Lubom\'ira Dvo\v r\'akov\'a}
\ead{lubomira.dvorakova@fjfi.cvut.cz}
\author[cvut]{Edita Pelantov\'a}
\ead{edita.pelantova@fjfi.cvut.cz}

\address[cvut]{FNSPE Czech Technical University in Prague, Czech Republic}

\date{September 2025}


\begin{abstract} Frequency of letters in a~symbolic sequence ${\bf u}$ over a finite alphabet is one of the basic characteristics of ${\bf u}$. The notion of $k$-balancedness captures the property that the number of any letter occurring in two arbitrary factors of ${\bf u}$ of equal length differs at most by $k$. For a fixed integer $k$ and alphabet size $d\in \mathbb N$, we discuss possible frequencies of letters in $k$-balanced $d$-ary sequences. For the size $d$ of the alphabet,  we introduce the notion of balancedness threshold $BT(d)$ and give an upper bound on it, where $BT(d)$ is the minimum $k$ such that there exists a $k$-balanced sequence over a $d$-letter alphabet for all possible letter frequencies.
\end{abstract}

\maketitle

\section{Introduction}
This paper is devoted to the study of the relation between frequencies of letters and $k$-balancedness in sequences (also called infinite words). Let us first introduce these notions.
Consider a sequence $\uu = u_0u_1u_2\cdots $ of symbols from a finite alphabet $\mathcal{A}=\{1,2,\ldots, d\}$.
The \textit{frequency of a~letter}~$a$ in $\uu$ is the limit (if it exists)
$$
f_a=\lim_{n\to \infty }\frac{\# \{i< n: u_i = a\}}{n}.
$$
If every letter $a \in \A$ has a~well-defined frequency in $\uu$, then $\sum_{a \in \A}f_a = 1$.
The vector $\vec{f}_\uu = (f_a)_{a\in \A}$ is called the \textit{frequency vector} of $\uu$.

\medskip
A \textit{word} $w=w_0w_1\cdots w_{n-1}$ over $\mathcal A$ is a finite sequence of letters $w_i$ from $\mathcal A$. Its \textit{length} $|w|$ equals $n$. To denote the number of occurrences of a letter $a$ in $w$, we use $|w|_a$. The set of all words over the alphabet $\mathcal A$ (including the empty word) is denoted $\mathcal A^*$.
A word $w$ is a  \textit{factor} of $\uu=u_0 u_1 u_2\cdots$ if there exists $i \in \mathbb N$ such that $w=u_i u_{i+1}\cdots u_{i+|w|-1}$.
We say that the sequence $\uu$ is recurrent if every factor of $\uu$ occurs in $\uu$ infinitely many times. The \textit{language} $\mathcal{L}(\uu)$ of a sequence $\uu$ is the set of factors occurring in $\uu$.
The \textit{factor complexity} of a sequence $\uu$ is the mapping ${\mathcal C}:\mathbb N \to \mathbb N$, where $${\mathcal C}(n)=\#\{w \in {\mathcal L}(\uu) \ : \ |w|=n\}\,.$$
We say that $\uu$ is $k$-\textit{balanced} if for any two factors $v, w$ of $\uu$ of the same length and for any letter $a\in \A$ holds $||v|_a-|w|_a|\leq k$. Obviously, a~$k$-balanced sequence is $K$-balanced for any $K>k$.

The study of $1$-balanced sequences over a~binary alphabet $\{a,b\}$ was initiated by Hedlund and Morse~\citep{MoHe40}.
They showed that $1$-balancedness requires some particular properties of the sequence. If a~$1$-balanced sequence is eventually periodic, then $f_a$ and $f_b$ are rational.  In the opposite case, a~$1$-balanced sequence is called \textit{Sturmian}. Hedlund and Morse proved that for each positive vector $(f_a,f_b)$, where $f_a+f_b = 1$, there exists a~$1$-balanced sequence with such letter frequencies.
The class of Sturmian sequences is the most studied class of sequences and there exist a lot of equivalent definitions of Sturmian sequences, see~~\citep{DvPeSt2010, Abelian2010, Vuillon2001}. These equivalent definitions enable to generalize Sturmian sequences to larger alphabets in many different ways, see~\citep{DvPeSt2010}.
On the one hand, one of the most usual generalizations are Arnoux-Rauzy sequences, however, it is known that the letter frequencies of any
Arnoux-Rauzy sequence belong to the Rauzy gasket~\citep{ArSt13}, which is a fractal set of Lebesgue measure zero.
On the other hand, for any given letter frequencies one can construct a sequence of sublinear factor complexity by coding a~$d$-interval exchange transformation.  It is however known~\citep{Zorich1997} that such generalization of Sturmian sequences to~$d$-ary alphabet is almost always unbalanced.  



$1$-balanced sequences over alphabets of size $d$ were described by Hubert~\citep{Hubert2000}. The description implies that for $d\geq 3$,  the frequency vector $(f_a)_{a\in\A}$ of $1$-balanced sequences takes only a particular form, see Lemma~\ref{lem:pouzeNeco}.
This fact motivates our definition of balancedness threshold for an alphabet of size $d$.

\begin{definition}
We say that $k\in \N$ is \textit{frequency restrictive} for $d$ if there exists a~positive vector $\vec{f} =(f_1,f_2, \ldots, f_d)$ with $\sum_{i\in \A} f_i=1$ such that no $k$-balanced $d$-ary sequence has the frequency vector~$\vec{f}$. \\
\textit{Balancedness threshold} $BT(d)$ is the minimum $k\in \N$ such that $k$ is not frequency restrictive for $d$.
\end{definition}
Obviously,  $BT(1)=0$. It follows from the result by Hedlund and Morse that $BT(2)=1$. In Lemma~\ref{lem:pouzeNeco} we explain why $BT(d)\geq 2$ for every $d\geq 3$. In~\citep{DMP} the sequence coding rectangle exchange transformation is used to prove $BT(3)=2$. Moreover, the factor complexity of such ternary sequences satisfies $\mathcal{C}(n) \leq  \alpha n^2\bigl(1+ o(1)\bigr)\,,$ for some parameter $\alpha \in (0,1)$. 
The equality $BT(3)=2$ follows also from the properties of ternary hypercubic billiard sequences. This class contains sequences of any given letter frequencies. They are $2$-balanced~\citep{AnVi} and under some additional condition on momentum, their factor complexity equals $n^2+n+1$, see~\citep{AR1994, ARTokyo1994, Be2007}.

The main contribution of this paper is the upper bound on balancedness threshold.

\begin{theorem}\label{thm:main} Let $d$ be a positive integer. Then $BT(d)\leq \lceil \log_2 d\rceil$.
\end{theorem}
On top of it, we discuss factor complexity of $d$-ary sequences used in the proof of the above theorem. Getting a lower bound on $BT(d)$ is beyond our means.  It is not clear whether for some $d\in \N$ holds $BT(d)\geq 3$. The characterization of $2$-balanced sequences, which would be helpful for this purpose, is still missing.

\section{Frequencies and balancedness in fixed points of morphism}
In combinatorics on words, sequences with some required properties are usually looked up among morphic sequences. Let us recall what is known on letter frequencies and balancedness of such sequences. It clarifies why such sequences cannot be used in the proof of Theorem~\ref{thm:main}.
Given two alphabets $\A$, $\B$, then a~\textit{morphism} is a map $\psi: \A^* \to \B^*$ such that $\psi(uv) = \psi(u)\psi(v)$ for all words $u, v \in \A^*$, where $uv$ means concatenation of the words $u$ and $v$.
The morphism $\psi$ can be naturally extended to a~sequence $\uu=u_0 u_1 u_2\cdots$ over $\A$ by setting
$\psi(\uu) = \psi(u_0) \psi(u_1) \psi(u_2) \cdots\,$.

A~\textit{fixed point} of a morphism $\psi:  \A^* \to  \A^*$ is a~sequence $\uu$ such that $\psi(\uu) = \uu$.
We associate to a~morphism $\psi: \A^* \to  \A^*$ the \textit{incidence matrix} $M_\psi$ defined for each $i,j \in \{1,2,\dots, d\}$ as $(M_\psi)_{ij}=|\psi(j)|_{i}$.
A~morphism $\psi$ is \textit{primitive} if the matrix $M_\psi$ is primitive, i.e., there exists $k\in \mathbb N$ such that $M_\psi^k$ is a positive matrix.


In the sequel, we limit our consideration to the fixed points of primitive morphisms.
Frequencies of letters in any fixed point of a primitive morphism are given by the coordinates of a unique positive eigenvector of norm one corresponding to the spectral radius~\citep{Qu}.
Therefore, the letter frequencies belong to an algebraic field of order at most $d$. The same is true for morphic sequences, i.e., morphic images of fixed points.
The letters in any morphic sequence $\uu$ have the so-called \textit{uniform letter frequencies}~\citep{Qu}: for any sequence $(k_n)_{n \in \N}$ of non-negative integers, the limit
$$
\lim_{n\to \infty }\frac{\# \{k_n\leq i< k_n+n: u_i = a\}}{n}
$$
exists and is the same for any choice of $(k_n)_{n \in \N}$. The definition of letter frequency in the introduction is restricted to the study of the limit for the sequence $(k_n) = (0)$.

The relation of balancedness and uniform letter frequency was described by Berthé and Delecroix~\citep{BeDe}.
\begin{theorem} A~sequence $\uu$ over an alphabet $\mathcal A$ is $k$-balanced for some $k\in \mathbb N$ if and only if it has uniform letter frequencies
and there exists a~constant $B$ such that for any factor $w$ of $\uu$, we have $\Bigl||w|_a - f_a|w|\Bigr|\leq B$   for every letter $a\in
\A$.
\end{theorem}

In general, proving that a~sequence $\uu$ is $k$-balanced for some $k$ may be a complicated problem. It is even more difficult to determine the minimum such $k$.
If $\uu$ is a~fixed point of a~primitive morphism, then it is possible to decide about $k$-balancedness using a result by Adamczewski~\citep{Ad}. It says that if all eigenvalues but one of the incidence matrix lie in the interior of the unit ball centered at origin, then $\uu$ is $k$-balanced. Moreover, an algorithm computing the minimum value $k$ is provided ibidem.

Since the letter frequencies in morphic sequences are always algebraic numbers, they cannot cover all possible candidates for $\vec{f}$.   

A very important generalization of morphic sequences is represented by $S$-adic systems, see for example~\citep{BeDe}. An $S$-adic system introduced by Cassaigne in~\citep{Ca} allows construction of ternary sequences with prescribed letter frequencies that are almost always $k$-balanced for some constant $k$, as proven in~\citep{CaLaLe}.

\section{Colouring of sequences}
In this section, we describe a construction that enables us to create sequences with prescribed letter frequencies.

\begin{definition} Let $\uu=u_0u_1u_2\cdots$ be a~sequence over the alphabet $\{a,b\}$.
Denote by $\mathcal{O}^{(a)}_n$ and $\mathcal{O}^{( b)}_n$ the $n^{th}$ occurrence of the letters $ a$ and $ b$ in $\uu$, respectively.
Let
${\bf a} = a_0a_1a_2\cdots$  and ${\bf b} = b_0b_1b_2\cdots$
 be two sequences over two disjoint alphabets $\mathcal{A}$  and $\mathcal{B}$, respectively. Colouring of $\uu$  by ${\bf a}$ and ${\bf b}$ is a sequence $\vv = v_0v_1v_2\cdots$ over $\mathcal{A}\cup \mathcal{B}$  such that for every   $N \in \N$   the $N^{th}$ entry of ${\bf v}$ is
$$v_N = \left\{ \begin{array}{cc} a_n,  & \text{if\ \ } N = \mathcal{O}^{(a)}_n;\\
b_n,  & \text{if\ \ } N = \mathcal{O}^{(b)}_n.
\end{array}\right. \quad$$
 We denote  ${\bf v}  = \colour(\uu, {\bf a},{\bf b}) $.
\end{definition}

Less formally: $\vv$ is obtained from the sequence $\uu$ over $\{ { a,b} \}$ by replacing the  letters ${ a}$'s in $\uu$ step by step by entries of the sequence $a_0a_1a_2 \cdots$ and analogously, the letters ${ b}$'s in $\uu$ are replaced by entries of the sequence $b_0b_1b_2\cdots$.

For ${\bf v}  = \colour(\uu, {\bf a},{\bf b})$ we use the notation $\pi(\vv) = \uu$ and $\pi(v) = u$ for any $v \in \LL(\vv)$ and the corresponding $u \in \LL(\uu)$.
We say that $\uu$ (resp. $u$) is a~\emph{projection} of $\vv$ (resp. $v$).
The map $\pi : \LL(\vv) \to \LL(\uu)$ is clearly a~morphism.
\begin{example}\label{ex:colouring} Let $\uu$ be a~sequence over $\{a,b\}$ and ${\bf a}=(121314)^{\omega}$ and ${\bf b}=(56)^{\omega}$, then
$$
\begin{array}{rcl}
\uu &=& aabaababaabaababaababaabaababaabaab\cdots \\

{\bf v}&=&12513615416215361451621531645126135  \cdots
\end{array}
$$
We have $\pi({54162153614})= baabaababaa$.
\end{example}

\begin{remark}\label{rem:freq}
Frequencies of letters in  $\vv = \colour(\uu, {\bf a}, {\bf b})$ can be easily computed from frequencies of letters in $\uu$, ${\bf a}$ and ${\bf b}$: if they exist.  For example, the letter $j \in \mathcal{A}$  has in $\vv$ the frequency $f_a\gamma$, where  $f_a$ is the frequency of the letter ${a}$ in $\uu$ and $\gamma$ is the frequency of the letter $j$ in ${\bf a}$.
\end{remark}

\begin{definition}
A~sequence $\aa$ is a \emph{constant gap sequence} if for each letter $a$ occurring in $\aa$ the distance between any consecutive occurrences of $a$ in $\aa$ is  constant.
\end{definition}
Example~\ref{ex:colouring} shows constant gap sequences $\aa$ and $\bb$.

The next result comes from~\citep{Newman}.
\begin{lemma}\label{lem:constant_gap_periods}
Let $\aa$ be a~constant gap sequence over an alphabet $\mathcal A$ containing more than one letter. Then $\aa$ contains two distinct letters having the same frequency.
\end{lemma}

\begin{theorem}[{\citep{Hubert2000}}]\label{thm:1balanced}
A recurrent aperiodic sequence $\vv$ is 1-balanced if and only if $\vv = \colour( \uu, \aa, \bb)$ for some Sturmian sequence $\uu$ and constant gap sequences ${\aa}, {\bb}$ over two disjoint alphabets.
\end{theorem}
Example~\ref{ex:colouring} shows a~$1$-balanced sequence $\vv$.

Using Remark~\ref{rem:freq}, we can describe the form of the frequency vector in $1$-balanced ternary and quaternary sequences.
\begin{observation}
Let $\vv=\colour(\uu, \aa, \bb)$ be a~1-balanced sequence over an alphabet $\{1,2,\dots, d\}$, where $\alpha$ is the frequency of $a$ in $\uu$. Then the frequency vector $\vec f_{\vv}$ takes on the following values.
\begin{enumerate}
\item For $d=3$, we have only one frequency vector $\vec f_\vv=(\alpha\frac{1}{2}, \alpha\frac{1}{2}, 1-\alpha)$ (up to some letter permutations) corresponding to $\aa=(12)^{\omega}$ and $\bb=(3)^{\omega}$.
\item For $d=4$, we have three possibilities (up to some letter permutations):
\begin{itemize}
\item if $\aa=(12)^{\omega}$, $\bb=(34)^{\omega}$, then $$\vec f_{\vv}=(\alpha\tfrac{1}{2}, \alpha\tfrac{1}{2}, (1-\alpha)\tfrac{1}{2}, (1-\alpha)\tfrac{1}{2})\,;$$
\item if $\aa=(123)^{\omega}, \bb=(4)^{\omega}$, then $$\vec f_{\vv}=(\alpha\tfrac{1}{3}, \alpha\tfrac{1}{3}, \alpha\tfrac{1}{3}, 1-\alpha)\,;$$
\item if $\aa=(1213)^{\omega}, \bb=(4)^{\omega}$, then $$\vec f_{\vv}=(\alpha\tfrac{1}{2}, \alpha\tfrac{1}{4}, \alpha\tfrac{1}{4}, 1-\alpha)\,.$$
\end{itemize}
\end{enumerate}
\end{observation}

The following lemma implies that $BT(d)\geq 2$ for $d\geq 3$.
\begin{lemma}\label{lem:pouzeNeco}
Let $\vv$ be a~$d$-ary 1-balanced sequence, where $\vv=\colour(\uu, \aa, \bb)$, $\uu$ is a Sturmian sequence over $\{a,b\}$ and $\aa$, $\bb$ are constant gap sequences over disjoint alphabets $\mathcal A$ and $\mathcal B$.
If $d\geq 3$, then $\vv$ contains two distinct letters of the same frequency.
\end{lemma}
\begin{proof}
Denote $\alpha$ the frequency of the letter $a$ in $\uu$.
If $d\geq 3$, then either $\#\mathcal A\geq 2$ or $\#\mathcal B\geq 2$. Consequently, by Lemma~\ref{lem:constant_gap_periods} either $\aa$ or $\bb$ contains two distinct letters $i,j$ such that they have the same frequency in $\aa$, resp. $\bb$, say $\gamma$. Then by Remark~\ref{rem:freq}, $f_i=f_j=\alpha\gamma$, resp. $f_i=f_j=(1-\alpha)\gamma$ are the frequencies of letters $i$ and $j$ in $\vv$.
\end{proof}

As a consequence of Lemma\ref{lem:pouzeNeco}, we can see that the frequency vectors of 1-balanced sequences do not take on all possible values.

\begin{lemma}\label{plus1} Let  $\uu$ be an $\ell$-balanced sequence over $\{ a,b\}$, and
${\bf a} = a_0a_1a_2\cdots$  and ${\bf b} = b_0b_1b_2\cdots$
 be two $k$-balanced sequences over two disjoint alphabets $\mathcal{A}$ and $\mathcal{B}$, respectively. Then  ${\bf v} = \colour(\uu, {\bf a},{\bf b}) $ is $(k+\ell)$-balanced.
\end{lemma}

\begin{proof}
Let $u, v$ be factors of $\bf v$ of the same length. We want to prove that for each letter $c \in  \mathcal{A}\cup \mathcal{B}$
$$\bigl||u|_c-|v|_c\bigr| \leq k+\ell \,.$$
WLOG let $c \in \mathcal A$.
Denote $u'=\pi(u)$ and $v'=\pi(v)$. Clearly $|u'|=|v'|$. Thanks to $\ell$-balancedness of $\uu$, we have $\bigl||u'|_{a}-|v'|_{ a}\bigr|\leq \ell$.
Let $\pi_{\mathcal A}:(\mathcal{A}\cup \mathcal{B})^*\to {\mathcal{A}}^*$ be a morphism such that $\pi_{\mathcal A}(x)=x$ if $x \in \mathcal A$ and $\pi_{\mathcal A}(x)=\varepsilon$ if $x \in \mathcal B$.
It holds that for each factor $w$ of $\bf v$ the word $\pi_{\mathcal A}(w)$ is a factor of $\bf a$.
By definition of $\pi_{\mathcal A}$, we have $|u|_c=|\pi_{\mathcal A}(u)|_c, \ |v|_c=|\pi_{\mathcal A}(v)|_c$ and by definition of colouring $|\pi_{\mathcal A}(u)|=|u'|_{ a}, \ |\pi_{\mathcal A}(v)|=|v'|_{ a}$.

Since $|u'|_{a}$ and $|v'|_{a}$ differ at most by $\ell$, the words $\pi_{\mathcal A}(u)$ and $\pi_{\mathcal A}(v)$ are factors of $\bf a$ whose lengths differ at most by $\ell$.

WLOG assume $|\pi_{\mathcal A}(u)|=|\pi_{\mathcal A}(v)|+n$, where $0\leq n\leq \ell$. Then $\pi_{\mathcal A}(u)=a_i\cdots a_{i+m+n}$ and $\pi_{\mathcal A}(v)=a_j\cdots a_{j+m}$ for some $i,j,m \in \mathbb N$. Then using $k$-balancedness of $\bf a$ we get
   $ \bigl||\pi_{\mathcal A}(u)|_c-|\pi_{\mathcal A}(v)|_c\bigr| $
    $$\begin{array}{rcl}
    &\leq & \bigl||a_i\cdots a_{i+m}|_c-|a_j\cdots a_{j+m}|_c\bigr|+\\
    && \hspace{2cm} +|a_{i+m+1}\cdots a_{i+m+n}|_c \\
    &\leq & k+n\leq k+\ell\,.
    \end{array}$$
Since $|u|_c=|\pi_{\mathcal A}(u)|_c$ and $|v|_c=|\pi_{\mathcal A}(v)|_c$, we have proved that $\bigl||u|_c-|v|_c\bigr| \leq k+\ell$.
\end{proof}

\section{Proof of Theorem \ref{thm:main} }
In this section, we prove a~statement having Theorem~\ref{thm:main} as its direct consequence.
We make use of knowledge on the number of occurrences of letters in Sturmian sequences, provided in~\citep{Lothaire}.

\begin{lemma}\label{freqExists} Let $\uu$ be a~1-balanced sequence over the alphabet $\{a,b\}$ and $f_a =\alpha \in (0,1)$.  Then any factor $u$ of length $n \in \N$ either contains $\lceil \alpha n \rceil$ letters $a$ and $\lfloor (1-\alpha) n \rfloor$ letters $b$, or $u$  contains  $\lfloor  \alpha n \rfloor$ letters $a$  and  $\lceil (1-\alpha) n \rceil$ letters $b$.
\end{lemma}

\begin{theorem}\label{thm:freq+complexity} Let  $d \in \N$, $d\geq 1$, and $f(1),f(2), \ldots, f(d)$ be  positive numbers such that $f(1)+f(2)+\cdots + f(d)=1$. Then there exists an infinite sequence  ${\bf v}$ over the alphabet $\{1,2,\ldots,d\}$ such that
\begin{enumerate}
    \item the frequency of the letter $i$ in ${\bf v}$ is $f(i)$ for each $i\in \{1,2,\ldots, d\}$;
    \item ${\bf v}$ is $k$-balanced with $k =\lceil \log_2 d\rceil$;
    \item the factor complexity of $\vv$  satisfies $$\mathcal{C}_\vv(n)\leq (n+1)^{d-1}\,.$$
\end{enumerate}
\end{theorem}
\begin{proof} We proceed by induction on $d$. If $d=1$, then we put ${\bf v} = 1^\omega$.

Let $d\geq 2$.  We denote $\alpha = f(1)+f(2) +\cdots +f(\lceil \frac{d}{2}\rceil)$ and
$$f'(i) =\left\{ \begin{array}{cl}\frac{1}{\alpha}\, f(i) &\text{ for } i=1,2, \ldots ,\lceil \frac{d}{2}\rceil\,;\\
& \\
\frac{1}{1-\alpha}\,f(i) &\text{ for }i =\lceil \frac{d}{2}\rceil +1, \ldots, d\,.
\end{array} \right.$$
 Obviously, $$ f'(1)+\cdots + f'(\lceil \tfrac{d}{2}\rceil)\ =\ 1\ =\  f'(\lceil \tfrac{d}{2}\rceil+1)+\cdots + f'(d).$$

By induction hypothesis, there exist a  $k_a$-balanced sequence ${\bf a} = a_0a_1a_2\cdots$ over the  alphabet $\mathcal{A} = \{1,2, \ldots, \lceil \tfrac{d}{2}\rceil\}$ with the frequencies of letters $f'(i)$ for each $i \in \mathcal{A}$  and $k_a = \lceil \log_2 \lceil \tfrac{d}{2}\rceil\rceil $ and a $k_b$-balanced sequence ${\bf b} = b_0b_1b_2\cdots$ over the  alphabet $\mathcal{B} = \{\lceil \tfrac{d}{2}\rceil +1, \ldots, d\}$ with the frequencies of letters $f'(i)$ for each $i \in \mathcal{B}$  and $k_b = \lceil \log_2 \lfloor \tfrac{d}{2}\rfloor\rceil $.

Let ${\bf u}$ be a $1$-balanced sequence over the alphabet $\{ a, b\}$ with frequencies of letters $\alpha$ and $1-\alpha$, respectively.
Then the sequence  ${\bf v} = \colour({\bf u},{\bf a},{\bf b})$  is over the alphabet $\{1, \ldots, d\}$, and  by Remark~\ref{rem:freq}, the frequencies of letters are $\alpha f'(i) = f(i)$ for $i \in \mathcal{A}$ and $(1-\alpha)f'(i) = f(i)$  for $i \in \mathcal{B}$.

Lemma \ref{plus1} implies that ${\bf v} = \colour({\bf u},{\bf a},{\bf b})$  is $k$-balanced, with $k= 1+\lceil \log_2 \lceil \tfrac{d}{2}\rceil\rceil $.
To complete the proof of Item 2, we have to show that
$1+\lceil \log_2 \lceil \tfrac{d}{2}\rceil\rceil \leq \lceil \log_ 2 d \rceil  $. If $d$ is even, then $1+\lceil \log_2 \lceil \tfrac{d}{2}\rceil\rceil  = \lceil \log_ 2 d \rceil  $.

For $d$ odd, we show the required inequality by contradiction. Assume that   $1+\lceil \log_2 \lceil \tfrac{d}{2}\rceil\rceil > \lceil \log_ 2 d \rceil  $. As $d$ is odd, we know that
$$\log_2 d < \lceil \log_2 d\rceil\leq \lceil \log_2 \lceil \tfrac{d}{2}\rceil\rceil=$$
$$\lceil \log_2  \tfrac{d+1}{2}\rceil=\lceil \log_2  (d+1)\rceil -1 < \log_2 (d+1) \,.$$
Therefore we have
$d<2^{\lceil \log_2 d\rceil}< d+1.$
The numbers $d, 2^{\lceil \log_2 d\rceil}$ and $d+1$ are integers, which leads to a~contradiction.

\medskip

To show Item 3, we proceed again by induction on $d$.
Let  $u$ be a fixed factor of length $n$ in the  $1$-balanced sequence $\uu$. The frequencies of letters in $\uu$ are $\alpha$ and $(1-\alpha)$. By Lemma \ref{freqExists}, the factor  $u$  either contains $\lceil \alpha n \rceil$ letters $a$ and  $\lfloor  (1-\alpha) n \rfloor$ letters $b$,  or $u$  contains  $\lfloor  \alpha n \rfloor$ letters $a$  and  $\lceil (1-\alpha) n \rceil$ letters $b$. Hence the  factor $u$  equals the  projection $\pi(v)$ for at most $\mathcal{C}_{\bf a}(\lceil \alpha n \rceil) \times\mathcal{C}_{\bf b}(\lceil (1-\alpha) n \rceil)$  factors $v$ in $\vv$.  Since $\uu$ has at most $n+1$ factors of length $n$, we have
$$\mathcal{C}_\vv(n)\leq (n+1) \, \mathcal{C}_{\bf a}(\lceil \alpha n \rceil) \,\mathcal{C}_{\bf b}(\lceil (1-\alpha) n \rceil).$$
Using the induction hypothesis and simple inequalities   $\lceil \alpha n \rceil \leq n$  and $\lceil (1-\alpha) n \rceil\leq n$, we conclude
$$ \mathcal{C}_\vv(n)\leq (n+1) \, \mathcal{C}_{\bf a}(n) \,\mathcal{C}_{\bf b}( n) \leq$$
$$\leq (n+1)\, (n+1)^{\lceil \tfrac{d}{2}\rceil-1} \, (n+1)^{\lfloor \tfrac{d}{2}\rfloor -1}  = (n+1)^{d-1}\,.
$$
\end{proof}

\section{Comments and questions}

\begin{enumerate}

    \item The bound we found on the factor complexity of the sequence $\vv$ constructed in the proof of Theorem~\ref{thm:freq+complexity} is not optimal. What is the optimal upper bound?

    \item  A $1$-balanced binary sequence $\uu$ is either Sturmian or periodic. If some ratio $f(i):f(j)$ in the assumptions of Theorem~\ref{thm:freq+complexity} is  rational,  we can reduce the degree of the polynomial in the upper bound $(n+1)^{d-1}$ at least by 1.

    \item When all frequencies $f(i)$ in the assumptions of Theorem~\ref{thm:freq+complexity} are rational, our construction gives a~periodic sequence $\vv$. How to determine its period?

    \item Given rational frequencies $f(i) = \frac{p_i}{q_i}$, how many steps are needed to construct a prefix of $\vv$ of length $N$?

\item A \emph{cubic billiard sequence in dimension} $d$ (also called hypercubic billiard sequence) is a~coding of the sequence of
the faces successively hit by a~billiard ball moving inside the unit hypercube $[0,1]^d$, where two
parallel faces are encoded by the same letter. These $d$-ary sequences are parameterized by the initial
position $x \in  [0, 1]^d$ and the initial momentum $\theta \in R^d \setminus \{0\}$ of the ball.
The vector of letter frequencies corresponds to the vector of initial momentum, up to a~dilatation, and a~change in the
signs of some components.
Under some additional condition on momentum, the factor complexity of cubic billiard sequences in dimension $d$ satisfies ${\mathcal C}(n)=n^{d-1}(1+o(1))$, see~\citep{Be09}. 
Vuillon~\citep{Vuillon} proved that any cubic billiard sequence in dimension $d$, whose momentum has rationally independent
components, is $d-1$ balanced. 
Andrieu and Vivion~\citep{AnVi} further specified that 
\begin{itemize}
\item for $d\in \{1,2,
\ldots, 4\}$,   any cubic billiard sequence in dimension $d$ generated by a momentum
with rationally independent components is not $d-2$ balanced; 
\item for $d\geq 5$,  cubic billiard sequences in dimension $d$ generated by a~momentum with
rationally independent components  are  $C$-balanced with the smallest $C \in \{3,4,\ldots, d-1\}$. 
\end{itemize}

\end{enumerate}

\section{Acknowledgements}
Both authors are aware of how significantly Prof. Havlíček influenced their scientific careers. As a teacher, as a co-author, and most importantly as a person who, in his role as dean of the faculty and head of the Department of Mathematics, encouraged his colleagues to pursue scientific work.



\end{document}